\documentclass[12pt,a4paper]{article}
\usepackage{enumerate}
\usepackage{amsmath}
\usepackage{amsthm}
\usepackage{amsfonts}
\usepackage{amssymb}
\usepackage[english]{babel}

\numberwithin{equation}{section}
\everymath{\displaystyle}
\title{Multitime hybrid differential games with \\ curvilinear integral functional}
\author{Constantin Udri\c ste, Elena-Laura Otob\^{i}cu, Ionel \c Tevy}
\begin{document}
\date{}
\renewcommand{\abstractname}{}

\maketitle
\numberwithin{equation}{section}
\newtheorem{theorem}{Theorem}[section]
\newtheorem{lemma}[theorem]{Lemma}
\newtheorem{corollary}[theorem]{Corollary}
\newtheorem{definition}[theorem]{Definition}
\newtheorem{remark}[theorem]{Remark}

\begin{abstract}
Multitime differential games are related to the modeling and analysis
of cooperation or conflict in the context of a multitime dynamical systems. Their theory
involves either a curvilinear integral functional or a multiple integral functional and
an $m$-flow as constraint.
The aim of this paper is to give original results regarding multitime hybrid differential games with
curvilinear integral functional constrained by an $m$-flow: fundamental properties of multitime upper and lower values,
viscosity solutions of multitime (HJIU) PDEs, representation formula of viscosity
solutions for multitime (HJ) PDEs, and max-min representations.
\end{abstract}



\section{Multitime hybrid differential game with \\curvilinear integral functional}

Let $t=(t^\alpha)\in \Omega_{0T}\subset \mathbb{R}^m_+$, $\alpha =1,...,m,$ be an evolution multi-parameter, called multitime. Consider an arbitrary $C^1$ curve $\Gamma_{0T}$ joining the diagonal opposite points $0=(0,\ldots,0)$ and $T=(T^1,\ldots,T^m)$ in the $m$-dimensional parallelepiped $\Omega_{0T}=[0,T]$ (multitime interval) in $\mathbb{R}^m_+$ endowed with the product order, a $C^2$ state vector $x:\Omega_{0T}\rightarrow \mathbb{R}^n, x(t)=(x^i(t)),$ $i=1,...,n$, a $C^1$ control vector $u(t)=(u_\alpha(t)):\Omega_{0T}\rightarrow U\subset \mathbb{R}^{qm},$ for the first equip of $m$ players (who wants to maximize), a $C^1$ control vector  $v(t)=(v_\alpha(t)):\Omega_{0T}\rightarrow V\subset \mathbb{R}^{qm},$ for the second equip of $m$ players (who wants to minimize), $u_\alpha(\cdot)=\Phi(\cdot,\eta_1(\cdot)), v_\alpha(\cdot)=\Psi(\cdot,\eta_2(\cdot)),$ a running cost $L_\alpha(t,x(t),u_\alpha(t),v_\alpha(t))dt^\alpha$ as a nonautonomous closed Lagrangian $1$-form (satisfies $D_\beta L_\alpha=D_\alpha L_\beta),$ a terminal cost $g(x(T))$ and the $C^1$ vector fields $X_\alpha=(X_\alpha^i)$ satisfying the complete integrability conditions (CIC)
 $D_{\beta}X_{\alpha}=D_{\alpha}X_{\beta}$ (m-flow type problem).

%%%%%%%%%%%%%%%%%%
In our paper, a {\it multitime hybrid differential game} is given by a multitime dynamics,
as a PDE system controlled by two controllers (first equip, second equip)
and a target including a curvilinear integral functional. The approach we follow below is those in the paper
\cite{[2]}, but we must be more creative since our theory is multitemporal one (see also \cite{[8]}-\cite{[21]}).
More precisely, we introduce and analyze a multitime differential game whose Bolza payoff is the sum between a
path independent curvilinear integral (mechanical work) and a function of the final event (the terminal cost, penalty term),
and whose evolution PDE is an m-flow: {\it Find
$$\min_{v(\cdot)\in V}\max_{u(\cdot)\in U} J(u(\cdot),v(\cdot))=\int_{\Gamma_{0T}} L_\alpha (s,x(s),u_\alpha(s),v_\alpha(s))ds^\alpha+g(x(T)),$$
subject to the Cauchy problem
$$\frac{\partial x^i}{\partial s^\alpha}(s)=X^i_\alpha(s,x(s),u_\alpha(s),v_\alpha(s)),$$
$$x(0)=x_0, \,\,s\in \Omega_{0T}\subset \mathbb{R}_+^m, \,\,x\in \mathbb{R}^n.$$}

Let $D_{\alpha}$ be the total derivative operator and $[X_{\alpha},X_{\beta}]$ be the
bracket of vector fields. Suppose the piecewise complete integrability conditions (CIC)
$$ \left( \frac{\partial X_{\alpha}}{\partial u^a_\lambda}\delta^{\gamma}_{\beta} -  \frac{\partial X_{\beta}}{\partial u^a_\lambda}\delta^{\gamma}_{\alpha}\right)\frac{\partial u^a_\lambda}{\partial s^{\gamma}}+\left( \frac{\partial X_{\alpha}}{\partial v^b_\lambda}\delta^{\gamma}_{\beta} -  \frac{\partial X_{\beta}}{\partial v^b_\lambda}\delta^{\gamma}_{\alpha}\right)\frac{\partial v^b_\lambda}{\partial s^{\gamma}}=\left[ X_{\alpha},X_{\beta}\right] + \frac{\partial X_{\beta}}{\partial s^{\alpha}} - \frac{\partial X_{\alpha}}{\partial s^{\beta}},$$
where $a, b =1,...,q$, are satisfied throughout.

To simplify, suppose that the curve $\Gamma_{0T}$ is an increasing curve in the multitime interval $\Omega_{0T}$.
If we vary the starting multitime and the initial point, then we obtain a
larger family of similar multitime problems containing the functional
$$J_{x,t}(u(\cdot),v(\cdot))=\int_{\Gamma_{tT}} L_\alpha (s,x(s),u_\alpha(s),v_\alpha(s))ds^\alpha+g(x(T)),$$
and the evolution constraint
$$\frac{\partial x^i}{\partial s^\alpha}(s)=X^i_\alpha(s,x(s),u_\alpha(s),v_\alpha(s)),$$
$$x(t)=x, \,\,s\in \Omega_{tT}\subset \mathbb{R}_+^m,\,\, x\in \mathbb{R}^n.$$

We assume that each vector field
$X_\alpha:\Omega_{0T}\times \mathbb{R}^n\times U\times V\rightarrow \mathbb{R}^n$
is uniformly continuous, satisfying
$$\left\{\begin{array}{ll}
\Vert X_\alpha(t,x,u_\alpha,v_\alpha)\Vert\leqslant A_\alpha\\
\Vert X_\alpha(t,x,u_\alpha,v_\alpha) - X_\alpha(t,\hat{x},u_\alpha,v_\alpha)\Vert\leqslant A_\alpha\Vert x-\hat{x}\Vert,
\end{array}\right.$$
for some constant 1-form $A=(A_\alpha)$ and all $t\in \Omega_{0T}, \,\,x, \hat{x}\in \mathbb{R}^n, u\in U,v\in V.$

 Suppose the functions
$$g:\mathbb{R}^n\rightarrow \mathbb{R}, \quad L_\alpha:\Omega_{0T}\times \mathbb{R}^n\times U\times V\rightarrow \mathbb{R}$$
are uniformly continuous and satisfy the boundedness conditions
$$\left\{\begin{array}{ll}
\vert g(x)\vert\leqslant B\\
\vert g(x)-g(\hat{x})\vert\leqslant B\Vert x-\hat{x}\Vert,
\end{array}\right.$$

$$\left\{\begin{array}{ll}
\vert L_\alpha(t,x,u_\alpha,v_\alpha)\vert\leqslant C_\alpha\\
\vert L_\alpha(t,x,u_\alpha,v_\alpha)-L_\alpha(t,\hat{x},u_\alpha,v_\alpha)\vert\leqslant C_\alpha\Vert x-\hat{x}\Vert,
\end{array}\right.$$
for constant $1$-form $C=(C_\alpha)$ and all $t\in \Omega_{0T},\,\, x, \hat{x}\in \mathbb{R}^n,\,\, u\in U,\,\,v\in V.$

\begin{definition}
(i) The set $$\mathcal {U}(t)=\left\lbrace u_\alpha(\cdot):\mathbb{R}^m_+\rightarrow U \vert \ u_\alpha(\cdot) \mathrm{ \ is \ measurable \ and \ satisfies \ CIC}\right\rbrace $$ is called \textbf{the control set for the first equip of players}.
(ii) The set $$\mathcal {V}(t)=\left\lbrace v_\alpha(\cdot):\mathbb{R}^m_+\rightarrow V \vert \ v_\alpha(\cdot) \mathrm{ \ is \ measurable \ and \ satisfies \ CIC}\right\rbrace $$ is called \textbf{the control set for the second equip of players}.
\end{definition}

\begin{definition}
(i) A map $\Phi:\mathcal {V}(t)\rightarrow \mathcal {U}(t)$ is called \textbf{a strategy for the first equip of players}, if the equality $v(\tau)=\widehat{v}(\tau), t\leq \tau \leq s \leq T$ implies $\Phi[v](\tau)=\Phi[\widehat{v}](\tau).$
(ii) A map $\Psi:\mathcal {U}(t)\rightarrow \mathcal {V}(t)$ is called \textbf{a strategy for the second equip of players}, if the equality $u(\tau)=\widehat{u}(\tau), t\leq \tau \leq s \leq T$ implies $\Psi[u](\tau)=\Psi[\widehat{u}](\tau).$
\end{definition}

Let $ \mathcal{A}(t)$ be \textbf{\ the set of strategies for the first equip of players}
 and $\mathcal{B}(t)$ be \textbf{\ the set of strategies for the second equip of players}.

\begin{definition}
 (i) The function
$$m(t,x)=\min_{\Psi\in \mathcal{V}} \max_{u(\cdot)\in U} J_{t,x}( u(\cdot),\Psi[u](\cdot))$$ is called \textbf{the multitime lower value function}.
(ii) The function
$$M(t,x)=\max_{\Phi\in \mathcal{U}} \min_{v(\cdot)\in V} J_{t,x}(\Phi[v](\cdot),v(\cdot)) $$ is called \textbf{the multitime upper value function}.
\end{definition}

The multitime lower value function $m(t,x)$ and the multitime upper value function $M(t,x)$
are piecewise continuously differentiable (see below, the boundedness and continuity of the values functions).

\section{Properties of lower and upper values}

\begin{theorem}\textbf{(multitime dynamic programming optimality conditions)}
For each pair of strategies $(\Phi,\Psi),$ the lower and upper value functions
can be written respectively in the form
\begin{equation}\begin{split}
m(t,x)\ & =\min_{\Psi\in \mathcal{B}(t)} \max_{u_\alpha\in \mathcal{U}(t)}\bigg\{  \int_{\Gamma_{tt+h}} L_\alpha (s,x(s),u_\alpha(s),\Psi[u_\alpha](s))ds^\alpha \\& +m(t+h,x(t+h))\bigg\}
 \end{split}\end{equation}
and
\begin{equation}\begin{split}
M(t,x)\ &  =\max_{\Phi\in \mathcal{A}(t)}\min_{v_\alpha\in \mathcal{V}(t)}\bigg\{  \int_{\Gamma_{tt+h}} L_\alpha (s,x(s),\Phi [v_\alpha](s),v_\alpha(s))ds^\alpha \\ & +M(t+h,x(t+h))\bigg\},
\end{split}\end{equation}
for all $(t,x) \in \Omega_{tT}\times \mathrm{R}^n$ and all $h\in \Omega_{0T-t}.$
\end{theorem}

\begin{proof} First we recognize the Bellman principle
(we write the value of a decision problem at a certain point in multitime
in terms of the payoff from some initial choices and the value of the remaining
decision problem that results from those initial choices).

To confirm the first statement, we shall use the function
\begin{equation}\begin{split}
w(t,x)\ &=\min_{\Psi\in \mathcal{B}(t)} \max_{u_\alpha\in \mathcal{U}(t)}\bigg\{  \int_{\Gamma_{tt+h}} L_\alpha (s,x(s),u_\alpha(s),\Psi[u_\alpha](s))ds^\alpha \\ & +m(t+h,x(t+h))\bigg\}.
\end{split}\end{equation}

We will show that, for all $\varepsilon >0,$ the lower value function $m(t,x)$
will satisfies two inequalities, $m(t,x)\leq w(t,x)+2\varepsilon$ and
$m(t,x)\geq w(t,x)-3\varepsilon.$ Since $\varepsilon>0$ is arbitrary, it follows $m(t,x)=w(t,x).$

\begin{enumerate}[i)]

\item For $\varepsilon >0,$ there exists a strategy $\Upsilon\in \mathcal{B}(t)$ such that

\begin{equation}\begin{split}\label{eq.7}
w(t,x)\ &\geqslant \max_{u_\alpha\in \mathcal{U}(t)}\bigg\{  \int_{\Gamma_{tt+h}} L_\alpha (s,x(s),u_\alpha(s),\Upsilon[u_\alpha](s))ds^\alpha \\ & +m(t+h,x(t+h))\bigg\}-\varepsilon.
\end{split}\end{equation}

We shall use the state $x(\cdot)$ which solves the (PDE), with the initial condition $\overline{x}=x(t+h)$
(Cauchy problem) on the set $\Omega_{tT}\setminus \Omega_{tt+h},$
for each $\overline{x}\in \mathbb{R}^n$. We can write
\begin{equation}\begin{split}
m(t+h,\overline{x}) \ & =\min_{\Psi\in \mathcal{B}(t+h)} \max_{u_\alpha\in \mathcal{U}(t+h)}\bigg\{ \int_{\Gamma_{t+hT}} L_\alpha (s,x(s),u_\alpha(s),\Psi[u_\alpha](s))ds^\alpha  \\ & +g(x(T))\bigg\}.
\end{split}\end{equation}

Thus there exists a strategy $\Upsilon_{\overline{x}}\in \mathcal{B}(t+h)$ for which

\begin{equation}\begin{split}\label{eq.8}
 m(t+h,\overline{x})\ & \geqslant \max_{u_\alpha\in \mathcal{U}(t+h)}\bigg\{  \int_{\Gamma_{t+hT}} L_\alpha (s,x(s),u_\alpha(s),\Upsilon_{\overline{x}}[u_\alpha](s))ds^\alpha \\& + g(x(T))\bigg\}-\varepsilon.
\end{split}\end{equation}

Define a new strategy
$$\Psi\in\mathcal{B}(t), \Psi[u_\alpha](s)\equiv
\left\{\begin{array}{ll}
\Upsilon[u_\alpha](s) & s\in \Omega_{tt+h}\\
\Upsilon_{\overline{x}}[u_\alpha](s) & s\in \Omega_{tT}\setminus\Omega_{tt+h},
\end{array}\right.$$
for each control $u_\alpha\in \mathcal{U}(t).$ For any $u_\alpha\in \mathcal{U}(t)$,
replacing the inequality $\eqref{eq.8}$ in the inequality $\eqref{eq.7}$, we obtain
$$w(t,x)\geqslant   \int_{\Gamma_{tT}} L_\alpha (s,x(s),u_\alpha(s),\Psi[u_\alpha](s))ds^\alpha+g(x(T))-2\varepsilon.$$
Consequently
$$ \max_{u_\alpha\in \mathcal{U}(t)}\left\lbrace  \int_{\Gamma_{tT}} L_\alpha (s,x(s),u_\alpha(s),\Psi[u_\alpha](s))ds^\alpha+g(x(T))\right\rbrace\leq w(t,x)+2\varepsilon.$$

Hence
$$m(t,x)\leq w(t,x)+2\varepsilon.$$

\item On the other hand, there exists a strategy $\Psi\in\mathcal{B}(t)$ for which we can write the inequality

\begin{equation}\label{eq.9}
m(t,x)\geqslant \max_{u_\alpha\in \mathcal{U}(t)}\left\lbrace  \int_{\Gamma_{tT}} L_\alpha (s,x(s),u_\alpha(s),\Psi[u_\alpha](s))ds^\alpha+g(x(T))\right\rbrace-\varepsilon.
\end{equation}

By the definition of $w(t,x),$ we have
\begin{equation}\begin{split}
w(t,x) \ & \leqslant \max_{u_\alpha\in U(t)}\bigg\{  \int_{\Gamma_{tt+h}} L_\alpha (s,x(s),u_\alpha(s),\Psi[u_\alpha](s))ds^\alpha \\& +m(t+h,x(t+h))\bigg\}
\end{split}\end{equation}
and consequently there exists a control $u^1_\alpha\in \mathcal{U}(t)$ such that
\begin{equation}\begin{split}\label{eq.10}
w(t,x) \ & \leqslant   \int_{\Gamma_{tt+h}} L_\alpha (s,x(s),u^1_\alpha(s),\Psi[u^1_\alpha](s))ds^\alpha  \\&  +m(t+h,x(t+h))+\varepsilon.
\end{split}\end{equation}

Define a new control $${u_\alpha^\star}\in \mathcal{U}(t), {u_\alpha^\star}(s)\equiv
\left\{\begin{array}{ll}
u^1_\alpha(s) & s\in \Omega_{tt+h}\\
 u_\alpha(s) & s\in \Omega_{tT}\setminus\Omega_{tt+h},
\end{array}\right.$$
for each control $u_\alpha\in \mathcal{U}(t+h)$ and then define the strategy ${\Psi}^\star\in\mathcal{B}(t+h), \Psi^\star[u_\alpha](s)\equiv\Psi[{u_\alpha^\star}](s), s\in \Omega_{tT}\setminus\Omega_{tt+h}.$
We find the inequality
\begin{equation}\begin{split}
 \ & m(t+h,x(t+h)) \\ & \leq\max_{u_\alpha\in \mathcal{U}(t+h)}\left\lbrace \int_{\Gamma_{tt+h}} L_\alpha(s,x(s),u_\alpha(s),\Psi^\star[u_\alpha](s))ds^\alpha+g(x(T))\right\rbrace
\end{split}\end{equation}
and so there exists the control $u^2_\alpha\in \mathcal{U}(t+h)$ for which
\begin{equation}\begin{split}\label{eq.11}
 \ & m(t+h, x(t+h))\\&  \leq \int_{\Gamma_{tT}\setminus \Gamma_{tt+h}} L_\alpha(s,x(s),u^2_\alpha(s),\Psi^\star[u^2_\alpha](s))ds^\alpha+g(x(T)) +\varepsilon.
\end{split}\end{equation}

Define a new control
$$u_\alpha\in \mathcal{U}(t), u_\alpha(s)\equiv
\left\{\begin{array}{ll}
u^1_\alpha(s) & s\in \Omega_{tt+h}\\
u^2_\alpha(s) & s\in \Omega_{tT}\setminus\Omega_{tt+h}.
\end{array}\right.$$

Then the inequalities $\eqref{eq.10}$ and $\eqref{eq.11}$ yield
$$ w(t,x)\leq \int_{\Gamma_{tT}} L_\alpha(s,x(s),u_\alpha(s),\Psi[u_\alpha](s))ds^\alpha+g(x(T)) +2\varepsilon,$$
and so $\eqref{eq.9}$ implies the inequality
$$ w(t,x)\leq m(t,x)+3\varepsilon.$$

This inequality and  $m(t,x)\leq w(x,t)+2\varepsilon$ complete the proof.
\end{enumerate}
\end{proof}

\begin{theorem}\textbf{(boundedness and continuity of the values functions)}
The lower, upper value function $m(t,x)$, $M(t,x)$ satisfy the boundedness conditions
$$\vert m(t,x)\vert, \vert M(t,x)\vert\leq D$$
$$\vert m(t,x)-m(\hat{t},\hat{x})\vert, \vert M(t,x)-M(\hat{t},\hat{x})\vert\leq E\,\, \ell (\Gamma_{\hat{t}\,t})+ F\,\Vert x-\hat{x}\Vert,$$
for some constants $D, E, F$ and for all $t, \hat{t}\in \Omega_{0T}, x, \hat{x} \in \mathbb{R}^n.$
\end{theorem}

\begin{proof} We prove only the statements for upper value function $M(t,x).$

Since $\vert g(x)\vert\leqslant B, \vert L_\alpha(t,x,u_\alpha,v_\alpha)\vert\leqslant C_\alpha, \alpha=\overline{1,m}$, we find
\begin{equation}\begin{split}
\vert J_{t,x}(u(\cdot),v(\cdot))\vert \ & =\Big\vert \int_{\Gamma_{tT}} L_\alpha (s,x(s),u_\alpha(s),v_\alpha(s))ds^\alpha+g(x(T))\Big\vert \\& \leq \Big\vert \int_{\Gamma_{tT}} L_\alpha (s,x(s),u_\alpha(s),v_\alpha(s))ds^\alpha \Big\vert + \vert g(x(T))\vert \\& \leq \int_{\Gamma_{tT}} \Vert L_\alpha (s,x(s),u_\alpha(s),v_\alpha(s))\Vert \Vert ds^\alpha \Vert +\vert g(x(T))\vert \\& \leq \Vert C \Vert \int_{\Gamma_{tT}}  ds + B= \Vert C \Vert l(\Gamma_{tT}) +B\leq \Vert C \Vert l(\Gamma_{0T}) +B=D \\ & \Longrightarrow \vert M(t,x)\vert\leq D,
\end{split}\end{equation}
for all $u_\alpha(\cdot)\in \mathcal{U}(t),v_\alpha(\cdot)\in \mathcal{V}(t).$

Let $x_1,x_2\in \mathbb{R}^n,\,\, t_1, t_2 \in \Omega_{0T}.$
For $\varepsilon>0$ and the strategy $\Phi\in \mathcal{A}(t_1),$ we have

\begin{equation}\label{eq:1}
M(t_1,x_1)\leq \min_{v_\alpha \in \mathcal{V}(t_1)} J(\Phi[v_\alpha],v_\alpha)+\varepsilon.
\end{equation}

Define the control
$$\overline{v}_\alpha\in \mathcal{V}(t_1),\overline{v}_\alpha(s)\equiv
\left\{\begin{array}{ll}
{v}_\alpha^1(s) & s\in \Omega_{0t_2}\setminus\Omega_{0t_1}\\
{v}_\alpha(s) & s\in \Omega_{0T}\setminus\Omega_{0t_2},
\end{array}\right.$$
for any $v_\alpha \in \mathcal{V}(t_2)$ and some $v_\alpha^1 \in V$ and for each $v_\alpha \in \mathcal{V}(t_2), \underline{\Phi}\in \mathcal{A}(t_2)$ (the restriction of $\Phi$ over $\Omega_{0T}\setminus \Omega_{0t_1})$ by $\underline{\Phi}[v_\alpha]=\Phi[\overline{v}_\alpha], s\in \Omega_{0T}\setminus\Omega_{0t_2}.$

Choose the control $v_\alpha\in \mathcal{V}(t_2)$ so that
\begin{equation}\label{eq:2}
M(t_2,x_2)\geq  J(\underline{\Phi}[v_\alpha],v_\alpha)-\varepsilon.
\end{equation}

By the inequality $\eqref{eq:1},$ we have
\begin{equation}\label{eq:3}
M(t_1,x_1)\leq  J(\Phi[\overline{v}_\alpha],\overline{v}_\alpha)+\varepsilon.
\end{equation}

We know that the (unique, Lipschitz) solution $x(\cdot)$ of the Cauchy problem
$$\left\{\begin{array}{ll}
\frac{\partial x^i}{\partial s^\alpha}(s)=X^i_\alpha(s,x(s),u_\alpha(s),v_\alpha(s))\\
x(t)=x, \quad s\in \Omega_{tT}\subset \mathbb{R}_+^m, x\in \mathbb{R}^n, i=\overline{1,n}, \alpha =\overline{1,m},
\end{array}\right.$$
is the response to the controls $u_\alpha(\cdot), v_\alpha(\cdot)$ for $s\in \Omega_{0T}.$

We choose $x_1(\cdot)$ as solution of the Cauchy problem
$$\left\{\begin{array}{ll}
\frac{\partial x^i_1}{\partial s^\alpha}(s)=X^i_\alpha(s,x_1(s),\Phi[\overline{v}_\alpha](s),\overline{v}_\alpha(s))\\
x_1(t_1)=x_1, \quad s\in \Omega_{0T}\setminus \Omega_{0t_1}.
\end{array}\right.$$
Equivalently, $x_1(\cdot)$ is solution of integral equation
$$x_1(s)= x_1(t_1) + \int_{\Gamma_{t_1s}}X_\alpha(\sigma,x_1(\sigma),\Phi[\overline{v}_\alpha](\sigma),\overline{v}_\alpha(\sigma))d\sigma^\alpha.$$
Take $x_2(\cdot)$ as solution of the Cauchy problem
$$\left\{\begin{array}{ll}
\frac{\partial x^i_2}{\partial s^\alpha}(s)=X^i_\alpha(s,x_2(s),\underline{\Phi}[v_\alpha](s),v_\alpha(s))\\
x_2(t_2)=x_2, \quad s\in \Omega_{0T}\setminus \Omega_{0t_2}.
\end{array}\right.$$
Equivalently, $x_2(\cdot)$ is solution of integral equation
$$x_2(s)= x_2(t_2) + \int_{\Gamma_{t_2s}}X_\alpha(\sigma,x_2(\sigma),\underline{\Phi}[v_\alpha](\sigma),\overline{v}_\alpha(\sigma))d\sigma^\alpha.$$

It follows that
$$\Vert x_1(t_2)-x_1 \Vert = \Vert x_1(t_2)-x_1(t_1)\Vert \leq \Vert A\Vert \,\ell(\Gamma_{t_1t_2}).$$

 Since $v_\alpha=\overline{v}_\alpha$ and $\underline{\Phi}[v_\alpha]=\Phi[\overline{v}_\alpha],$
 for $s\in \Omega_{0T}\setminus\Omega_{0t_2}$, we find the estimation
\begin{equation}\begin{split}
\Vert x_1(s)-x_2(s)\Vert \ &  \leq \Vert  x_1(t_1) - x_2(t_2)\Vert + \Vert \int_{\Gamma_{t_1t_2}}\cdots \Vert\\ & \leq
\Vert A \Vert \ell(\Gamma_{t_1t_2})+ \Vert x_1-x_2\Vert
 ,\,\, \hbox{on}\,\, t_2\leq s\leq T.
\end{split}\end{equation}

Thus the inequalities $\eqref{eq:2}$ and $\eqref{eq:3}$ imply
$$
 M(t_1,x_1)-M(t_2,x_2) \leq  J(\Phi[\overline{v}_\alpha],\overline{v}_\alpha])-J(\underline{\Phi}[v_\alpha],v_\alpha])+2\varepsilon
$$
$$\leq \Big\vert \int_{\Gamma_{t_1t_2}} L_\alpha (s,x_1(s),\Phi[\overline{v}_\alpha](s),\overline{v}_\alpha(s))ds^\alpha$$ $$+\int_{\Gamma_{t_2T}} (L_\alpha (s,x_1(s),\underline{\Phi}[v_\alpha](s),v_\alpha(s))  -L_\alpha (s,x_2(s),\underline{\Phi}[v_\alpha](s),v_\alpha(s)))ds^\alpha$$  $$+g(x_1(T))-g(x_2(T))+2\varepsilon\Big\vert
$$
$$\leq\int_{\Gamma_{t_1t_2}} \vert L_\alpha (s,x_1(s),\Phi[\overline{v}_\alpha](s),\overline{v}_\alpha(s))ds^\alpha\vert$$ $$+\int_{\Gamma_{t_2T}} \vert (L_\alpha (s,x_1(s),\underline{\Phi}[v_\alpha](s),v_\alpha(s))  -L_\alpha (s,x_2(s),\underline{\Phi}[v_\alpha](s),v_\alpha(s)))ds^\alpha\vert$$
$$+\vert g(x_1(T))-g(x_2(T))\vert +2\varepsilon$$
$$\leq \Vert C \Vert \ell(\Gamma_{t_1t_2}) +\Vert C \Vert \ell(\Gamma_{t_2T})\,(\Vert A \Vert \ell(\Gamma_{t_1t_2})+ \Vert x_1-x_2\Vert) +B\, \Vert x_1-x_2\Vert) +2\varepsilon$$
$$\leq \Vert C \Vert \ell(\Gamma_{t_1t_2}) +\Vert C \Vert \ell(\Gamma_{0T})\,(\Vert A \Vert \ell(\Gamma_{t_1t_2})+ \Vert x_1-x_2\Vert) +B\, \Vert x_1-x_2\Vert) +2\varepsilon.$$

Since $\varepsilon$ is arbitrary, we obtain the inequality
\begin{equation}\label{eq:7}
M(t_1,x_1)-M(t_2,x_2)\leq E\,\ell(\Gamma_{t_1t_2}) + F \,\Vert x_1-x_2\Vert.
\end{equation}

Let $\varepsilon>0$ and choose the strategy $\Phi\in \mathcal{A}(t_2)$ such that
\begin{equation}\label{eq:4}
M(t_2,x_2)\leq \min_{v_\alpha \in \mathcal{V}(t_2)} J(\Phi[v_\alpha],v_\alpha)+\varepsilon.
\end{equation}

For each control $v_\alpha \in \mathcal{V}(t_1)$ and $s\in \Omega_{0T}\setminus\Omega_{0t_2},$ define the control $\underline{v}_\alpha\in \mathcal{V}(t_2), \underline{v}_\alpha(s)=v_\alpha(s).$

For some $u^1_\alpha \in U,$ we define the strategy $\overline{\Phi}\in \mathcal{A}(t_1)$
(the restriction of $\Phi$ over $\Omega_{0T}\setminus\Omega_{0t_2}$) by
$$\overline{\Phi}[{v}_\alpha]=
\left\{\begin{array}{ll}
u^1_\alpha & s\in \Omega_{0t_2}\setminus\Omega_{0t_1}\\
\Phi[\underline{v}_\alpha] & s\in \Omega_{0T}\setminus\Omega_{0t_2}.
\end{array}\right.$$

Now choose a control $v_\alpha\in \mathcal{V}(t_1)$ so that
\begin{equation}\label{eq:5}
M(t_1,x_1)\geq  J(\overline{\Phi}[v_\alpha],v_\alpha)-\varepsilon.
\end{equation}

By the inequality $\eqref{eq:4},$ we have
\begin{equation}\label{eq:6}
M(t_2,x_2)\leq  J(\Phi[\underline{v}_\alpha],\underline{v}_\alpha)+\varepsilon.
\end{equation}

We choose $x_1(\cdot)$ as solution of the Cauchy problem (PDE system + initial condition)
$$\left\{\begin{array}{ll}
\frac{\partial x_1^i}{\partial s^\alpha}(s)=X^i_\alpha(s,x_1(s),\overline{\Phi}[v_\alpha],v_\alpha(s)), s\in \Omega_{0T}\setminus\Omega_{0t_1} \\
x_1(t_1)=x_1, \quad s\in \Omega_{0T}\setminus \Omega_{0t_1}
\end{array}\right.$$
and $x_2(\cdot)$ as solution of the Cauchy problem (PDE system + initial condition)
$$\left\{\begin{array}{ll}
\frac{\partial x_2^i}{\partial s^\alpha}(s)=X^i_\alpha(s,x_2(s),\Phi[\underline{v}_\alpha],\underline{v}_\alpha(s)), s\in \Omega_{0T}\setminus\Omega_{0t_2}\\
x_2(t_2)=x_2, \quad s\in \Omega_{0T}\setminus \Omega_{0t_2}.
\end{array}\right.$$
Using the associated integral equations, it follows that
$$\Vert x_1(t_2)-x_1 \Vert = \Vert x_1(t_2)-x_1(t_1)\Vert \leq \Vert A\Vert \,\ell(\Gamma_{t_1t_2}).$$
Also, for $s\in \Omega_{0T}\setminus\Omega_{0t_2}, v_\alpha=\underline{v}_\alpha$ and $\overline{\Phi}[v_\alpha]=\Phi[\underline{v}_\alpha],$ we find
\begin{equation}\begin{split}
\Vert x_1(s)-x_2(s)\Vert \ &  \leq \Vert  x_1(t_1) - x_2(t_2)\Vert + \Vert \int_{\Gamma_{t_1t_2}}\cdots \Vert\\ & \leq
\Vert A \Vert \ell(\Gamma_{t_1t_2})+ \Vert x_1-x_2\Vert,\,\, \hbox{on}\,\, t_2\leq s\leq T.
\end{split}\end{equation}

Thus, the relations $\eqref{eq:5}$ and $\eqref{eq:6}$ imply
$$
M(t_2,x_2)-M(t_1,x_1) =  J(\overline{\Phi}[v_\alpha],v_\alpha])-J(\Phi[\underline{v}_\alpha],\underline{v}_\alpha])+2\varepsilon
  $$
  $$
  =  - \int_{\Gamma_{t_1t_2}} L_\alpha (s,x_1(s),\overline{\Phi}[{v}_\alpha](s),{v}_\alpha(s))ds^\alpha
$$
 $$
+\int_{\Gamma_{t_2T}} (L_\alpha (s,x_1(s),{\Phi}[\underline{v}_\alpha](s),\underline{v}_\alpha(s)) - L_\alpha (s,x_2(s),{\Phi}[\underline{v}_\alpha](s),\underline{v}_\alpha(s)))ds^\alpha
 $$
 $$ +g(x_1(T))-g(x_2(T))+2\varepsilon $$
 $$\leq \Vert C \Vert \ell(\Gamma_{t_1t_2}) +\Vert C \Vert \ell(\Gamma_{0T})\,(\Vert A \Vert \ell(\Gamma_{t_1t_2})+ \Vert x_1-x_2\Vert) +B\, \Vert x_1-x_2\Vert) +2\varepsilon.$$

 Since $\varepsilon$ is arbitrary, we obtain the inequality
\begin{equation}\label{eq:8}
M(t_2,x_2)-M(t_1,x_1)\leq E\,\ell(\Gamma_{t_1t_2}) + F \,\Vert x_1-x_2\Vert.
\end{equation}

By $2.17$ and $2.22$, we proved the continuity of the lower and upper value functions.
\end{proof}

\section{Viscosity solutions of \\multitime (HJIU) PDEs}

\begin{theorem}\textbf{(PDEs for multitime upper value function, resp. multitime lower value function)}

The multitime upper value function $M(t,x)$ and the multitime lower value function $m(t,x)$
are the viscosity solutions of Hamilton-Jacobi-Isaacs-Udri\c ste (HJIU) PDEs:
\begin{itemize}
\item the multitime upper (HJIU) PDEs
$$\frac{\partial M}{\partial t^\alpha}(t,x)+\min_{v_\alpha\in \mathcal{V}} \max_{u_\alpha\in \mathcal{U}} \left\lbrace  \frac{\partial M}{\partial x^i}(t,x) X_\alpha^i(t,x,u_\alpha,v_\alpha)+L_\alpha(t,x,u_\alpha,v_\alpha)\right\rbrace =0,$$
with the terminal condition $M(T,x)=g(x),$
\item the multitime lower (HJIU) PDEs
$$\frac{\partial m}{\partial t^\alpha}(t,x)+\max_{u_\alpha \in \mathcal{U}} \min_{v_\alpha \in \mathcal{V}} \left\lbrace  \frac{\partial m}{\partial x^i}(t,x) X_\alpha^i(t,x,u_\alpha,v_\alpha)+L_\alpha(t,x,u_\alpha,v_\alpha)\right\rbrace =0,$$
with the terminal condition $m(T,x)=g(x).$
\end{itemize}
\end{theorem}

\begin{remark}
If we introduce the so-called upper and lower Hamiltonian $1$-forms defined respectively by
$$H^+_\alpha(t,x,p)=\min_{v_\alpha\in \mathcal{V}} \max_{u_\alpha \in \mathcal{U}}\lbrace p_i(t) X_\alpha^i(t,x,u_\alpha,v_\alpha)+L_\alpha(t,x,u_\alpha,v_\alpha)\rbrace,$$
$$H^-_\alpha(t,x,p)=\max_{u_\alpha\in \mathcal{U}} \min_{v_\alpha\in \mathcal{V}}\lbrace p_i(t) X_\alpha^i(t,x,u_\alpha,v_\alpha)+L_\alpha(t,x,u_\alpha,v_\alpha)\rbrace,$$
then the multitime (HJIU) PDE systems can be written in the form
$$\frac{\partial M}{\partial t^\alpha}(t,x)+H^+_{\alpha}\left( t,x,\frac{\partial M}{\partial x}(t,x)\right) =0$$ and
$$\frac{\partial m}{\partial t^\alpha}(t,x)+H^-_\alpha\left( t,x,\frac{\partial m}{\partial x}(t,x)\right) =0.$$
\end{remark}

The proof will be given in another paper.

\section{Representation formula of viscosity \\solutions for multitime (HJ) PDEs}

In this section, we want to obtain a representation formula for the viscosity solution
$M(t,x)$ of the multitime (HJ) PDEs system
\begin{equation}
 \frac{\partial M}{\partial t^\alpha}+H_\alpha\left( t,x,\frac{\partial M}{\partial x}(t,x)\right) =0, (t,x)\in \Omega_{0T}\times \mathbb{R}^n,\alpha=\overline{1,m},
\end{equation}
\begin{equation}
M(0,x)=g(x), x\in \mathbb{R}^n \,(\hbox{initial\, condition}),
\end{equation}
where the unique solution $M(t,x)$ satisfies the inequalities
\begin{equation}\label{eq:11}
\left\{\begin{array}{ll}
\vert M(t,x)\vert\leq D\\
\vert M(t,x)-M(\hat{t},\hat{x})\vert\leq E\,\,\ell(\Gamma_{t\hat t})+ F\,\,\Vert x-\hat{x}\Vert,
\end{array}\right.\end{equation}
for some constants $D, E, F$ (for $m=1,$ see also [4]).

Also, we assume that $g:\mathbb{R}^n \rightarrow \mathbb{R}, H_\alpha:\Omega_{0T} \times \mathbb{R}^n\times \mathbb{R}^p\rightarrow \mathbb{R},$ satisfy the inequalities
$$\left\{\begin{array}{ll}
\vert g (x)\vert\leq B\\
\vert g(x)-g (\hat{x})\vert\leq B \Vert x-\hat{x}\Vert
\end{array}\right.$$
and
\begin{equation}\label{eq:8}
\left\{\begin{array}{ll}
\vert H_\alpha(t,x,0)\vert\leq K_\alpha\\
\vert H_\alpha (t,x,p)-H_\alpha(\hat{t},\hat{x},\hat{p})\vert\leq K_\alpha\,\, (\ell(\Gamma_{t\hat t})+\Vert x-\hat{x}\Vert +\Vert p-\hat{p}\Vert).
\end{array}\right. \end{equation}

\textbf{Max-min representation of a Lipschitz function as affine functions} (for $m=1,$ see also [2], [3]).

\begin{lemma}\label{l-2}

For each $\alpha$, let
\begin{equation}\label{eq:10}
\left\{\begin{array}{ll}
U=B(0,1)\subset \mathbb{R}^{n}\\
V=B(0,P)\subset \mathbb{R}^{n}\\
X_\alpha(u_\alpha)=K_\alpha u_\alpha,\, K_\alpha \in \mathbb{R}\\
L_\alpha(t,x,u_\alpha,v_\alpha)=H_\alpha (t,x,v_\alpha)-<K_\alpha u_\alpha,v_\alpha>.
\end{array}\right.
\end{equation}
Let $H_\alpha$ be a Lipschitz 1-form. For some constant $P>0$ and for each $t\in \Omega_{0T},\,\, x \in \mathbb{R}^n,$ we have
$$H_\alpha (t,x,{p})=\max_{v_\alpha \in V}\min_{u_\alpha \in U}\left\lbrace <X_\alpha(u_\alpha),{p}> + L_\alpha(t,x,u_\alpha,v_\alpha)\right\rbrace ,$$
if $\Vert {p}\Vert \leq P$.
\end{lemma}

\begin{proof}
In view of the assumption $H_\alpha (t,x,v_\alpha)-H_\alpha(t,x,{p})\leq K_\alpha \Vert {p}-v_\alpha\Vert,$
by the Cauchy-Schwarz formula, and by the condition $||u||\leq 1$,
we have for any $x\in \mathbb{R}^n,$
\begin{equation}\begin{split}
H_\alpha (t,x,{p})\ & =\max_{v_\alpha \in V} \left\lbrace H_\alpha(t,x,v_\alpha) - K_\alpha\Vert {p}-v_\alpha\Vert\right\rbrace \\ & =\max_{v_\alpha \in V}\min_{u_\alpha\in U}\left\lbrace H_\alpha(t,x,v_\alpha)\,+\, <K_\alpha u_\alpha,{p}-v_\alpha>\right\rbrace.
\end{split}\end{equation}
\end{proof}

\textbf{Max-min representation of a Lipschitz function as positive homogeneous functions} (for m=1, see also [2],[3]).

\begin{lemma}
Let $H_\alpha$ be a Lipschitz $1$-form which is homogeneous in $p,$ i.e.,
$$H_\alpha(t,x,\lambda p)=\lambda H_\alpha(t,x,p),\,\,\lambda \geq 0.$$
Then there exist compact sets $U, V\subset \mathbb{R}^{2n}$ and
vector fields
$$X_\alpha:[0,T]\times \mathbb{R}^n\times U\times V\rightarrow \mathbb{R}^n$$
satisfying
$$\Vert X_\alpha(x)-X_\alpha(\hat{x})\Vert\leqslant A_\alpha\Vert x-\hat{x}\Vert$$
and such that, for each $\alpha$,
$$H_\alpha (t,x,p)=\max_{v_\alpha \in V}\min_{u_\alpha \in U}\left\lbrace <X_\alpha(t,x,u_\alpha,v_\alpha),p> \right\rbrace ,$$
for all $t\in \Omega_{0T},x\in \mathbb{R}^n,p \in \mathbb{R}^n.$
\end{lemma}

\begin{proof}
Let $u_\alpha=(u^1_\alpha,u^2_\alpha),v_\alpha=(v^1_\alpha,v^2_\alpha)$ ($2n$-dimensional controls) and
\begin{equation}\label{eq:9}
\left\{\begin{array}{ll}
U=V=B(0,1)\times B(0,1)\subset \mathbb{R}^{2n}\\
L_\alpha(t,x,u^1_\alpha,v^1_\alpha)=H_\alpha (t,x,v_\alpha^1)-<K_\alpha u^1_\alpha,v_\alpha^1>\\
X_\alpha(t, x, u_\alpha, v_\alpha)=K_\alpha u^1_\alpha+ C_\alpha v^2_\alpha+ (L_\alpha(t,x,u^1_\alpha,v^1_\alpha)- C_\alpha)u^2_\alpha.\\
\end{array}\right.
\end{equation}
According to Lemma $\eqref{l-2}$ and the assumptions $\eqref{eq:9},$ if $\Vert\eta\Vert =1,$ we have
\begin{equation}\begin{split}
H_\alpha (t,x,\eta)\ &  =\max_{v^1_\alpha \in V^1}\min_{u^1_\alpha\in U^1}\left\lbrace <K_\alpha u^1_\alpha,\eta> + L_\alpha (t,x,u^1_\alpha, v^1_\alpha)\right\rbrace,
\end{split}\end{equation}
for $U^1=V^1=B(0,1)\in \mathbb{R}^n$.

For any $p\neq 0$, we can write
\begin{equation}\begin{split}
H_\alpha(t,x,p)\ & =\Vert p\Vert H_\alpha \left( t,x,\frac{p}{\Vert p\Vert}\right) \\ & =\max_{v^1_\alpha\in V^1}\min_{u^1_\alpha\in U^1}\left\lbrace <K_\alpha u^1_\alpha,p>+L_\alpha (t,x,u^1_\alpha,v^1_\alpha)\Vert p\Vert\right\rbrace .
\end{split}\end{equation}
Then, if we choose $C_\alpha>0$ such that $\vert L_\alpha\vert\leq C_\alpha,$ we find
\begin{equation}\begin{split}
H_\alpha (t,x,p)\ &  =\max_{v^1_\alpha \in V^1}\min_{u^1_\alpha\in U^1}\bigg\{ <K_\alpha u^1_\alpha,p> +C_\alpha\Vert p\Vert +( L_\alpha (t,x,u^1_\alpha, v^1_\alpha)-C_\alpha)\Vert p\Vert\bigg\}\\ & =\max_{v^1_\alpha \in V^1}\min_{u^1_\alpha\in U^1}\max_{v^2_\alpha \in V^1}\min_{u^2_\alpha\in U^1}\bigg\{ <K_\alpha u^1_\alpha,p> +<C_\alpha v^2_\alpha, p>\\ &  +( L_\alpha (t,x,u^1_\alpha, v^1_\alpha)-C_\alpha)< u^2_\alpha,p> \bigg\} \\ & =\max_{v_\alpha \in V}\min_{u_\alpha\in U}\bigg\{ <X_\alpha(t,x,u_\alpha, v_\alpha),p>\bigg\}.
\end{split}\end{equation}
Now, interchanging $\min_{u_\alpha^1\in U^1}$ and $\max_{v_\alpha^2\in V^1}$, the result in Lemma follows.
\end{proof}

We are now in a position to give the main result of this section.
\begin{theorem}
For each $t \in \Omega_{0T}$ and $x \in \mathbb{R}^n,$ the upper value function $M(t,x)$ verifies the equality
\begin{equation}\begin{split}
M(t,x)=\max_{\Phi\in \mathcal{U}(T-t)}\min_{v_\alpha\in V(T-t)}\bigg\{ \ & - \int_{\Gamma_{T-tT}} L_\alpha(T-s,x(s),\Phi[v_\alpha](s),v_\alpha(s))ds^\alpha \\ & +g(x(T))\bigg\} ,
\end{split}\end{equation}
where for each pair of controls $v_\alpha\in V(T-t)$,\,\, $u_\alpha=\Phi[v_\alpha]\in U(T-t),$
the state function $x(\cdot)$ solves the problem
\begin{equation}
\left\{\begin{array}{ll}
\frac{\partial x^i}{\partial s^\alpha}(s)=-F^i_\alpha u_\alpha(s), s\in \Omega_{0T}\setminus \Omega_{0T-t}\\
x(T-t)=x.
\end{array}\right.
\end{equation}
\end{theorem}

\begin{proof}
Let
$$H^1_\alpha(t,x,p)=\max_{v_\alpha\in V}\min_{u_\alpha\in U}\left\lbrace <X_\alpha(u_\alpha),p>+L_\alpha(t,x,u_\alpha,v_\alpha)\right\rbrace,$$
 $U=B(0,1)\subset \mathbb{R}^{pm}, V=B(0,P)\subset \mathbb{R}^{qm}$ and $X^i_\alpha, L_\alpha$
 Lipschitz functions with the assumptions $\eqref{eq:10}.$

Then $H_\alpha(t,x,p)=H^1_\alpha(t,x,p)$ provided $\vert p\vert\leq P.$
Since $M(t,x)$ satisfies $\eqref{eq:11},$ it follows that $M(t,x)$ is also the unique viscosity
solution of the multitime (HJ) PDEs system (for $m=1,$ see also [4])
\begin{equation}
 \frac{\partial M}{\partial t^\alpha}+H^1_\alpha\left( t,x,\frac{\partial M}{\partial x}(t,x)\right) =0, \,\,(t,x)\in \Omega_{0T}\times \mathbb{R}^n,\alpha=\overline{1,m},
\end{equation}
\begin{equation}
M(0,x)=g(x),\,\, x\in \mathbb{R}^n.
\end{equation}
If we take $M^1(t,x)=M(T-t,x),$ one observes that $M^1(t,x)$ is a viscosity solution of this system (for $m=1,$ see also [2])

\begin{equation}
 \frac{\partial M^1}{\partial t^\alpha}+H^+_\alpha\left( t,x,\frac{\partial M^1}{\partial x}(t,x)\right) =0,\,\, (t,x)\in \Omega_{0T}\times \mathbb{R}^n,\alpha=\overline{1,m},
\end{equation}
\begin{equation}
M^1(T,x)=g(x),\,\, x\in \mathbb{R}^n
\end{equation}
and
$$H^+_\alpha(t,x,p)=\max_{v_\alpha\in V}\min_{u_\alpha\in U}\left\lbrace -<X_\alpha(u_\alpha),p>+L_\alpha(T-t,x,u_\alpha,v_\alpha)\right\rbrace.$$
Using the above developments, we obtain
\begin{equation}\begin{split}
M^1(t,x)=M(t,x)=\max_{\Phi\in \mathcal{U}(t)}\min_{v_\alpha\in V(t)}\bigg\{ \ & - \int_{\Gamma_{tT}} L_\alpha(T-s,x(s),\Phi[v_\alpha](s),v_\alpha(s))ds^\alpha \\ & +g(xT))\bigg\} ,
\end{split}\end{equation}
where $x(\cdot)$ is the solution of the Cauchy problem
\begin{equation}
\left\{\begin{array}{ll}
\frac{\partial x^i}{\partial s^\alpha}(s)=-X^i_\alpha(u_\alpha(s))=-F^i_\alpha u_\alpha(s), s\in \Omega_{0T}\setminus \Omega_{0T-t}\\
x(t)=x,
\end{array}\right.
\end{equation}
for the control $u_\alpha(\cdot)=\Phi[v_\alpha].$
\end{proof}


\begin{thebibliography}{99}
\bibitem{[1]} L. C. Evans, \textit{An Introduction to Mathematical Optimal Control Theory}, Lectures Notes, University of California, Departament of Mathematics, Berkeley, (2005).
\bibitem{[2]} L. C. Evans, P. E. Souganidis, \textit{Differential games and representation formulas for solutions of Hamilton-Jacobi-Isaacs equations}, Indiana University Mathematics Journal, 33, 5, (1984), 773-797.
\bibitem{[3]} M. G. Crandall, L. C. Evans, P. L. Lions, \textit{Some properties of viscosity solutions of Hamilton-Jacobi equations}, Trans. Amer. Math. Soc., 282, 2, (1984), 487-502.
\bibitem{[4]} M. G. Crandall, P. L. Lions, \textit{Viscosity Solutions of Hamilton-Jacobi equations}, Trans. Amer. Math. Soc., 277, (1983), 1-42.
\bibitem{[5]} E. N. Barron, L. C. Evans, R. Jensen, \textit{Viscosity solutions of Isaacs' equations and differential games with Lipschitz controls}, Journal of Differential Equations, 53, (1984), 213-233.
\bibitem{[6]} L. G\' omez Esparza, G. Mendoza Torres, L. M. Saynes Torres, \textit{A Brief Introduction to Differential Games}, International Journal of Physical and Mathematical Sciences, 4, 1, (2013).
\bibitem{[7]} G. Jank, \textit{Introduction to Non-cooperative Dynamical Game Theory}, Coimbra, (2001).
\bibitem{[8]} P. E. Souganidis, \textit{Existence of viscosity solutions of Hamilton-Jacobi equations}, Journal of Differential Equations, 56, (1985), 345-390.
\bibitem{[9]} K. Margellos, J. Lygeros, \textit{Hamilton-Jacobi formulation for reach-avoid differential games}, IEEE Trans. Automat. Contr., 56, 8, (2011), 1849-1861.
\bibitem{[10]} C. Udri\c ste, I. \c Tevy, \textit{Multi-time Euler-Lagrange-Hamilton theory}, WSEAS Trans. Math., 6, 6, (2007), 701-709.
\bibitem{[12]} C. Udri\c ste, \textit{Multitime stochastic control theory}, in Selected Topics on Circuits, Systems, Electronics, Control and Signal Processing, Proc. of the 6-th WSEAS International Conference on Circuits, Systems, Electronics, Control an Signal Processing (CSECS’07), Cairo, Egypt, December 29-31, (2007), 171-176.
\bibitem{[13]} C. Udri\c ste, \textit{ Multi-time controllability, observability and bang-bang principle},
J. Optim. Theory Appl., 138, 1 (2008), 141--157.
\bibitem{[14]} C. Udriste, L. Matei, I. Duca, {\it Multitime Hamilton-Jacobi Theory}, Proceedings of the 8th WSEAS International Conference on Applied Computer and Applied Computational Science, 509-513, 2009.
\bibitem{[15]} C. Udri\c ste, \textit{Equivalence of multitime optimal control problems}, Balkan J. Geom. Appl. 15, 1, (2010), 155-162.
\bibitem{[16]} C. Udri\c ste, \textit{Simplified multitime maximum principle}, Balkan J. Geom. Appl. 14, 1, (2009), 102-119.
\bibitem{[17]} C. Udri\c ste, I. \c Tevy, \textit{Multitime dynamic programming for curvilinear integral actions}, J. Optim. Theory and Appl., 146, (2010), 189-207.
\bibitem{[18]} C. Udri\c ste, L. Matei, \textit{Lagrange-Hamilton Theories} (in Romanian), Monographs and Textbooks 8, Geometry Balkan Press, Bucharest, (2008).
\bibitem{[19]} C. Udri\c ste, A. Bejenaru, \textit{Multitime optimal control with area integral costs on boundary}, Balkan J. Geom. Appl., 16, 2, (2011), 138-154
\bibitem{[20]} C. Udri\c ste, \textit{Multitime maximum principle for curvilinear integral cost}, Balkan J. Geom. Appl., 16, 1, (2011), 128-149.
\bibitem{[21]} C. Udri\c ste, I. \c Tevy, \textit{Multitime dynamic programming for multiple integral actions}, J. Glob. Optim., 51, 2, (2011), 345-360.
\bibitem{[22]} A. W. Starr, \textit{Nonzero-sum differential games: concepts and models}, Division of Engineering and Applied Physics Harvard University-Cambridge, Massachusetts, Technical Report, 590, (1969).
\bibitem{[23]} A. Davini, M. Zavidovique, \textit{On the (non) existence of viscosity solutions of multi-time Hamilton-Jacobi equations,} Preprint (2013), {http://www.math.jussieu.fr/~zavidovique/articles/NonCommutingNov2013.pdf}







\end{thebibliography}
\end{document}